\documentclass[11pt, oneside]{amsart} 

\usepackage[british,english]{babel} 
\usepackage{graphics,color,pgf}
\usepackage{epsfig}
\usepackage[ansinew]{inputenc}
\usepackage[all]{xy}
\usepackage{hyperref}
\newdir{ >}{!/8pt/@{}*@{>}}
\usepackage{amssymb, amsmath,amsthm, mathtools, amscd}
\usepackage{mathrsfs}
\usepackage{stmaryrd}
\usepackage[margin=1.5in]{geometry}

\makeatletter
\newtheorem*{rep@theorem}{\rep@title}
\newcommand{\newreptheorem}[2]{%
\newenvironment{rep#1}[1]{%
 \def\rep@title{#2 \ref{##1}}%
 \begin{rep@theorem}}%
 {\end{rep@theorem}}}
\makeatother

\newtheorem{intro_thm}{Theorem}

\theoremstyle{plain}
\newtheorem{teor}{Theorem}[section]
\newreptheorem{teor}{Theorem}  
\newtheorem{lem}[teor]{Lemma}

\newtheorem{prop}[teor]{Proposition}
\newreptheorem{prop}{Proposition}  

\theoremstyle{definition}
\newtheorem{deft}[teor]{Definition}

\theoremstyle{remark}
\newtheorem{oss}[teor]{Remark}

\DeclareMathOperator\homeo{\textup{Homeo}(X)}

\DeclareMathOperator\upH{\textup{H}}

\DeclareMathOperator\upL{\textup{L}}

\DeclareMathOperator\bbS{\mathbb{S}}

\DeclareMathOperator\bbZ{\mathbb{Z}}

\begin{document}

\title[Elementarity of virtual dendro-morphisms]{A note on elementarity of virtual dendro-morphisms for higher rank lattices}

\author[A. Savini]{A. Savini}
\address{Section de Math\'ematiques, University of Geneva, Rue du Conseil General 7-9, 1205 Geneva, Switzerland}
\email{Alessio.Savini@unige.ch}
\date{\today.\ \copyright{\ The author was partially supported by the SNSF grant no. 200020-192216.}}

\begin{abstract}
Let $\Gamma$ be a discrete countable group and let $(\Omega,\mu)$ be an ergodic standard Borel probability $\Gamma$-space. Given any non-elementary virtual dendro-morphism (that is a measurable cocycle in the automorphism group of a dendrite), we construct a unitary representation $V$ with no invariant vectors such that $\text{H}^2_b(\Gamma;V)$ contains a non-zero class. As a consequence, all virtual dendro-morphisms of a higher rank lattice must be elementary. 
\end{abstract}
  
\maketitle

\section{Introduction}

From a dynamical point of view, the most elementary and at the same time intriguing field of research is the study of group actions on topological spaces of dimension one. The very first case is given by \emph{circle actions}, that is representations of a group $\Gamma$ in $\textup{Homeo}^+(\bbS^1)$, the group of orientation preserving homeomorphisms. For such actions, orbits are well-understood. Moreover, one can construct a complete invariant for the dynamics of a circle action using the \emph{bounded Euler class} $e^b_{\mathbb{Z}} \in \upH^2_b(\textup{Homeo}^+(\bbS^1);\bbZ)$. More precisely, Ghys \cite{ghys:articolo} proved that two different circle actions $\rho_1,\rho_2:\Gamma \rightarrow \textup{Homeo}^+(\bbS^1)$ are equivalent (that is, \emph{semiconjugated}) if and only if the bounded classes obtained by pulling back $e^b_{\bbZ}$ are equal. More recently, Burger and Monod \cite{BM99,burger2:articolo} exploited some vanishing results in bounded cohomology to show that higher rank lattices act elementarily on the circle (see also \cite{Ghys1} for a different proof). 

Remaining in the context of one dimensional topological spaces, one can study \emph{dendrites}. Dendrites are continua containing no simple closed curve. Examples of dendrites are the ends compactification of a countable tree, some Julia sets \cite{BT07} and the Wazewski's universal dendrites \cite{Waz23}. Dendrites appear also as canonical quotients of continua \cite{Why28,Why30}. 

Monod and Duchesne studied systematically the dynamics of groups acting on dendrites. In \cite{DM19} they gave several structural properties about the homeomorphisms group $\homeo$, when $X$ is a dendrite containing no invariant subdendrite with respect to $\homeo$. When $\Gamma$ is a higher rank lattice or a lattice in a product of connected groups, the authors \cite{DM16} recovered an elementarity result similar to the one valid for circles actions. 

Using the same cohomological approach by Burger and Monod, the author \cite{Sav21} has recently extended Ghys' Theorem to the context of measurable cocycles. Measurable cocycles can be thought of as distorted actions where the distortion is parametrized by a standard Borel probability space. For this reason, since we are going to talk about measurable cocycles with values into the homeomorphisms group of a dendrite, we are actually going to talk about \emph{virtual dendro-morphism}. The term \emph{virtual} comes from the notation adopted by Mackey \cite{Mac66}.

Assuming that the virtual dendro-morphism is not \emph{elementary}, that is it does not admit an invariant family of points or arcs, we are able to extend some results by Monod and Duchesne in this context. More precisely, we have the following

\begin{intro_thm}\label{teor:unitary:rep}
Let $\Gamma$ be a countable group and let $(\Omega,\mu)$ be an ergodic standard Borel probability $\Gamma$-space. Given a non-elementary virtual dendro-morphism $\sigma:\Gamma \times \Omega \rightarrow \homeo$, there exists a canonical unitary $\Gamma$-representation $V$ without invariant vectors and such that $\textup{H}^2_b(\Gamma;V)$ has a non-zero class.
\end{intro_thm}

The unitary representation $V$ constructed above is given by some \emph{Bochner space} over the parameter space $\Omega$ which determines the distorted action. In this case, the reader can recognize similar techniques used by the author and Sarti \cite{Sav21,sarti:savini} to define the notion of \emph{parametrized class}. 

An immediate consequence of Theorem \ref{teor:unitary:rep} is that any virtual dendro-morphism of an amenable group is elementary. This extends the case of representations, known thanks to the work by Shi and Ye \cite{SY16}. 

Using jointly Theorem \ref{teor:unitary:rep} and the vanishing result by Monod and Shalom \cite[Theorem 1.4]{MonShal0} we are able to obtain the following

\begin{intro_thm}\label{teor:elementary}
Let $\mathbf{G}$ be a connected almost $k$-simple algebraic group over a field $\kappa$ with $\textup{rank}_{\kappa}(\mathbf{G}) \geq 2$. Let $\Gamma$ be a lattice in the $k$-points of $\mathbf{G}$. Any virtual dendro-morphism of $\Gamma$ over an ergodic standard Borel probability $\Gamma$-space is elementary. 
\end{intro_thm}

The above theorem should be compared with the ones by Witte-Zimmer \cite[Theorem 5.4]{WZ01}, Navas \cite[Theorem B]{Nav06} and the author \cite[Theorem 4]{Sav21}.

\subsection*{Plan of the paper} The first section is devoted to the basics about dendrites. In the second section we introduce virtual actions and we show that non-elementary ones admit a Furstenberg map. We conclude with the proof of the main theorems. 

\section{Dendrites}\label{sec:dendrites}

In this section we are going to remind the main properties of dendrites. For a more detailed discussion about those notions we refer the reader either to \cite[Section 2, Section 7]{DM16} or to \cite[Chapter X]{Nad92}. 

In topology, a \emph{continuum} is a non-empty connected compact metrizable space. A continuum $X$ is a \emph{dendrite} if any two distinct points of $X$ can be separated by a point. There exist other equivalent ways to define a dendrite. For instance, $X$ is a dendrite if and only if it is locally connected and contains no simple closed curve. 

Any non-empty closed connected set of a dendrite is still a dendrite. Since the intersection of two connected sets in $X$ remains connected \cite[Theorem 10.10]{Nad92}, given a non-empty subset $A$ it makes sense to consider the smallest subdendrite of $X$ containing $A$. We call it \emph{dendro-hull} of $A$ and we denote it by $[A]$. In the particular case when $A$ is the set $\{x,y\}$ of only two points, we use the notation $[x,y]$ to denote its dendro-hull. The subdendrite $[x,y]$ is called arc, and we define the \emph{interior} of the arc as $(x,y):=[x,y] \setminus \{x,y\}$. Notice that one can give another equivalent characterization of dendrites using arcs. Indeed for any two points $x,y \in X$, there exists a unique arc $[x,y]$ with $x,y$ as extremities. 

Since we are going to use it later, recall that given any subdendrite $M$ of $X$, there exists a retraction map $r^M:X \rightarrow M$, which is the identity on $M$ and such that $r(x)$ is the unique point contained in any arc from $x$ to a point in $M$, whenever $x \in X \setminus M$ \cite[Lemma 10.25,Terminology 10.26]{Nad92}.

Let $X$ be a dendrite and consider $x \in X$. Since $X$ is locally connected, the connected components of the complement $X \setminus \{x\}$ are open. The cardinality of the set of connected components is the (Menger-Urysohn) \emph{order} of $x$ in $X$ \cite[Section 46, I]{Kur61}. When $x$ has order equal to $1$ is called an \emph{end point} of $X$. We denote the subset of ends in $X$ by $\textup{Ends}(X)$. If the order of $x$ is $\geq 3$, then $x$ is a \emph{branch point} and we denote by $\textup{Br}(X)$ the set of branch points. Finally, if the order of $x$ is equal to $2$, we a have a \emph{regular point}. The set of regular points is dense in $X$, the one of end points is not empty and the branch points are countable. 

We conclude this short introduction about dendrites recalling the notion of fundamental bundle. Let $X$ be a dendrite not reduced to a point. The \emph{fundamental bundle} $\textup{Bund}(X)$ is the collection of pairs of the form $(x,C)$, where $x \in X$ and $C$ is a component of $X \setminus \{x\}$. Since the connected components are countably many, we can identify the fiber of $x$ to the discrete space $\pi_0(X \setminus \{x\})$. To topologize $\textup{Bund}(X)$ one can map it injectively in the product space $X \times \textup{Con(X)}$ with the map $(x,C) \mapsto (x, \{x\} \cup C)$. The space $\textup{Con}(X)$ is the set of closed connected subsets of $X$ endowed with the Vietoris topology (it is metrizable and compact). With the subspace topology, $\textup{Bund}(X)$ becomes a locally compact second countable space. 

One can also define the \emph{double bundle} $\textup{Bund}^2(X)$ over $X$ as the fibered product of two copies of $\textup{Bund}(X)$. The bundle $\textup{Bund}^2(X)$ has discrete countable fibers corresponding to pair of connected components of $X \setminus \{x \}$ for any $x \in X$. 

We conclude by noticing that $\homeo$ acts on $\textup{Bund}(X)$ and $\textup{Bund}^2(X)$ by homeomorphisms and all the projections are equivariant with respect to the $\homeo$-actions. 

\section{Virtual dendro-morphisms}

In this section we are going to introduce the notion of \emph{virtual dendro-morphism}. The latter will be a distorted group action on a dendrite and the distorsion will be parametrized by a standard Borel probability space. In the particular case of non-elementary virtual dendro-morphism we are going to show that there exists a unique minimal equivariant family of subdendrites. The existence of this minimal family will allow us to show the existence of a Furstenberg map for a non-elementary virtual dendro-morphism.

Given a dendrite $X$, recall that $\homeo$ is the group of homeomorphisms endowed with the compact-open topology. We can define the natural Borel structure associated to the latter topology, thus we will refer to such measurable structure when we are going to speak about measurable functions with target in $\homeo$. 

Before giving the definition of virtual dendro-morphism, recall that a standard Borel probability space $(\Omega,\mu)$ is a Borel space which is measurably isomorphic to a separable and completely metrizable space. If $\Gamma$ is a group acting on $(\Omega,\mu)$ by preserving the probability measure $\mu$, we are going to refer to $(\Omega,\mu)$ as a \emph{standard Borel probability $\Gamma$-space}. 

\begin{deft}\label{def:virtual:dendro:morphism}
Let $\Gamma$ be a countable group and let $(\Omega,\mu)$ be a standard Borel probability $\Gamma$-space. A \emph{virtual dendro-morphism} (or equivalently a \emph{measurable cocycle}) is a Borel map 
$$
\sigma:\Gamma \times \Omega \rightarrow \homeo 
$$ 
which satisfies 
\begin{equation}\label{eq:virtual:morphism}
\sigma(\gamma_1 \gamma_2,s)=\sigma(\gamma_1,\gamma_2.s)\sigma(\gamma_2,s) \ ,
\end{equation}
for every $\gamma_1,\gamma_2 \in \Gamma$ and almost every $s \in \Omega$. 
\end{deft}

Virtual dendro-morphisms are an actual generalization of group actions. Indeed any group action given by a representation $\rho:\Gamma \rightarrow \homeo$ can be seen as a virtual dendro-morphism by defining
$$
\sigma_\rho:\Gamma \times \Omega \rightarrow \homeo \ , \ \ \sigma(\gamma,s):=\rho(\gamma) \ ,
$$
for any standard Borel probability $\Gamma$-space $(\Omega,\mu)$. More generally, given a virtual dendro-morphism, one can construct a $\Gamma$-action on the product space $\Omega \times X$ determined by $\sigma$. Indeed we can define
\begin{equation}\label{eq:extended:action}
\gamma.(s,a):=(\gamma.s,\sigma(\gamma,s)a) \ ,
\end{equation}
for every $\gamma \in \Gamma, a \in X$ and almost every $s \in \Omega$. The cocycle condition given by Equation \eqref{eq:virtual:morphism} implies exactly that Equation \eqref{eq:extended:action} is an action. 

In this context the word cocycle refers to the fact that a virtual dendro-morphism is a cocycle in the sense of Feldman and Moore \cite{feldman:moore} associated to the orbital equivalence relation $\mathcal{R}_\Gamma$ of $\Gamma$ on $\Omega$. Following such interpretation, we can say that two virtual dendro-morphisms $\sigma_1,\sigma_2:\Gamma \times \Omega \rightarrow \homeo$ are \emph{equivalent} (or \emph{cohomologous}) if there exists a measurable map $f:\Omega \rightarrow \homeo$ such that
$$
\sigma_2(\gamma,s)=f(\gamma.s)^{-1}\sigma_1(\gamma,s)f(s) \ ,
$$
for every $\gamma \in \Gamma$ and almost every $s \in \Omega$. The equivalence between two virtual dendro-morphisms boils down to the fact that the associated actions on the product space $\Omega \times X$ are measurably conjugated by the map $(s,a) \mapsto (s,f(s)(a))$. 

Among all the possible virtual dendro-morphisms, there is the particular class of elementary ones. Before introducing them, we need to introduce the set $\mathcal{P}_{1,2}(X)$ of subsets of $X$ containing either only one or two points. Being a closed subset of $\textup{CL}(X)$, the set of all closed subsets of $X$ endowed with the Vietoris topology, we can consider on $\mathcal{P}_{1,2}(X)$ the Borel structure coming from the subspace topology. 

\begin{deft}\label{def:elementary:cocycle}
Let $\Gamma$ be a countable group and let $(\Omega,\mu)$ be a standard Borel probability $\Gamma$-space. Given a virtual dendro-morphism $\sigma:\Gamma \times \Omega \rightarrow \homeo$, we say that a Borel map 
$$
\varphi:\Omega \rightarrow \textup{CL}(X) 
$$ 
is \emph{$\sigma$-equivariant} if it holds
\begin{equation}\label{eq:sigma:eq}
\varphi(\gamma.s)=\sigma(\gamma,s)\varphi(s) \ .
\end{equation}
for every $\gamma \in \Gamma$ and almost every $s \in \Omega$. 

We say that the virtual dendro-morphism $\sigma$ is \emph{elementary} if there exists a $\sigma$-equivariant Borel map $\varphi:\Omega \rightarrow \mathcal{P}_{1,2}(X)$.
\end{deft}

In the particular case that $\sigma$ is an actual morphism with values into $\homeo$, it is immediate to verify that Definition \ref{def:elementary:cocycle} boils down to the definition of elementary action on a dendrite \cite[Definition 3.1]{DM16}. It is easy to verify that elementarity is a property which is invariant along the cohomology class of a fixed virtual dendro-morphism. Additionally, we want to remark that, if $\Omega$ is a $\Gamma$-ergodic space, then elementarity implies the existence of either an equivariant map $\varphi:\Omega \rightarrow X$ or an equivariat map $\varphi:\Omega \rightarrow \mathcal{P}_2(X)$, where the latter is the collection of $2$-points subsets of $X$. 

As for representations \cite[Proposition 3.2]{DM16}, there exist several characterizations of elementarity for a virtual dendro-morphism $\sigma$. A direct application of \cite[Proposition 3.3]{DM16} shows that the elementarity of $\sigma$ is equivalent to the existence of a measurable $\sigma$-equivariant map $\mu:\Omega \rightarrow \textup{Prob}(X)$, that is $\mu(\gamma.s)=\sigma(\gamma,s)_\ast\mu(s)$. Here $\textup{Prob}(X)$ is the space of probabilities over $X$ with the weak-$^\ast$ topology and $\sigma(\gamma,s)_\ast$ is the push-forward. Notice that the $\sigma$-equivariant map $\mu$ is nothing else that a $\Gamma$-fixed point for the affine action on $\textup{L}^0(\Omega,\textup{Prob}(X))$ defined by
$$
(\gamma.\mu)(s):=\sigma(\gamma^{-1},s)^{-1}_\ast \mu(\gamma^{-1}.s) \ ,
$$
for every $\gamma \in \Gamma$ and almost every $s \in \Omega$. Recall that $\upL^0(\Omega,\textup{Prob}(X))$ is the compact metrizable space of classes of measurable functions identified when they differ only on a null measure subset of $\Omega$.  

Viewing elementarity as a fixed point property, we can immediately argue that it can be extended using \emph{coamenability}. Recall that a subgroup $\Lambda < \Gamma$ is \emph{coamenable} if any continuous affine $\Gamma$-action on a convex compact set has a fixed point whenever it has a $\Lambda$-fixed point. Following the same line of \cite[Lemma 6.1]{DM16} we get the following

\begin{prop}\label{prop:coamenable}
Let $\Gamma$ be a countable group and let $\Lambda < \Gamma$ be a coamenable subgroup. Any virtual dendro-morphism of $\Gamma$ over a standard Borel probability $\Gamma$-space $(\Omega,\mu)$ is elementary if and only if its restriction to $\Lambda$ is elementary. 
\end{prop}

\begin{proof}
We prove only the non-trivial implication. Given a virtual dendro-morphism $\sigma:\Gamma \times \Omega \rightarrow \homeo$, suppose that its restriction to $\Lambda \times \Omega$ is elementary. This implies the existence of a $\Lambda$-fixed point in $\upL^0(\Omega,\textup{Prob}(X))$. Since the latter is convex and compact \cite[Lemma 7.1.(6)]{BFS}, by coamenability we get a $\Gamma$-fixed point. 
\end{proof}

In dynamics it is pretty usual to look for minimal invariant subsets for a given action. In our context we need first to introduce the notion of \emph{minimality} for closed-valued maps. We are going to follow Furstenberg \cite[Section 3]{furst:articolo}. Given $\sigma:\Gamma \times \Omega \rightarrow \homeo$, a $\sigma$-equivariant Borel map
$$
\kappa:\Omega \rightarrow \textup{CL}(X) 
$$
is \emph{minimal} if, given any other equivariant Borel map $\kappa':\Omega \rightarrow \textup{CL}(X)$, we have that 
\begin{equation}\label{eq:partial:order}
\kappa(s) \subseteq \kappa'(s) \ ,
\end{equation}
for almost every $s \in \Omega$. It is worth noticing that Equation \eqref{eq:partial:order} defines a partial order on the set of measurable closed-valued map on $\Omega$. We denote the space of (equivalence classes) of measurable closed-valued maps by $\textup{L}^0(\Omega,\textup{CL}(X))$. 

Given a minimal map $\kappa$ for $\sigma$, we are going to call the collection $\{\kappa(s)\}_{s \in \Omega}$ a \emph{minimal family} for $\sigma$. We are going to prove that for a non-elementary virtual dendro-morphism there exists a unique minimal family of subdendrites (compare with \cite[Lemma 4.1]{DM16}) . 

\begin{prop}\label{prop:minimal:family}
Let $\Gamma$ be a countable group and let $(\Omega,\mu_\Omega)$ be an ergodic standard Borel probability $\Gamma$-space. If a virtual dendro-morphism $\sigma:\Gamma \times \Omega \rightarrow \homeo$ is not elementary, then there exists a unique (as a class) equivariant minimal family of subdendrites for $\sigma$. 
\end{prop}

\begin{proof}
The existence is a mere consequence of Zorn's Lemma. The space $\textup{L}^0(\Omega,\textup{CL}(X))$ can be metrized using the following distance
$$
d(\kappa,\kappa'):=\int_{\Omega} d_{\textup{Haus}}(k(s),k'(s))d\mu(s) \ ,
$$
where $d_{\textup{Haus}}$ is the Hausdorff distance (notice that $X$ is compact and metrizable, hence the Hausdorff metric induces the Vietoris topology). Starting from the partial order defined by Equation \eqref{eq:partial:order}, we can exploit the compactness of $X$ and apply Zorn's Lemma to guarantee the existence of a minimal family $\mathscr{K}:=\{K_s\}_{s \in \Omega}$ of closed sets. To get a minimal family of subdendrites we can consider
$$
M_s:=[K_s] \ ,
$$
for almost every $s \in \Omega$. The family $\mathscr{M}:=\{M_s\}_{s \in \Omega}$ is $\sigma$-equivariant by the equivariance of the family $\mathscr{K}$. 

We want to show that the minimal family is unique. By contradiction, suppose that there exist two minimal equivariant families $\mathscr{M}=\{M_s\}_{s \in \Omega}$, $\mathscr{N}=\{N_s\}_{s \in \Omega}$ of subdendrites. By the ergodicity of the space $\Omega$, the two families $\mathscr{M}$ and $\mathscr{N}$ must be disjoint for almost every $s \in \Omega$ (otherwise they would intersect for almost every $s \in \Omega$ leading to an equivariant family of subdendrites strictly contained in both $\mathscr{M}$ and $\mathscr{N}$). 

If we consider the retraction map
$$
r^{\mathscr{M}}_s:X \rightarrow M_s 
$$
of Section \ref{sec:dendrites}, we can define the point $a_s:=r^{\mathscr{M}}_s(N_s)$. The equivariance of the families $\mathscr{M}$ and $\mathscr{N}$ and the uniqueness of the rectraction point imply that 
$$
a_{\gamma.s}=\sigma(\gamma,s)a_s \ ,
$$
for every $\gamma \in \Gamma$ and almost every $s \in \Omega$. In this way we obtained an equivariant measurable map $\alpha:\Omega \rightarrow X, \ \alpha(s):=a_s$, but this is a contradiction to the non-elementarity assumption. 
\end{proof}

Now we need to recall the notion of $\Gamma$-boundary for a countable group $\Gamma$. We are going to follow Bader and Furman \cite[Definition 2.1]{BF14}. 

Recall that, given a locally compact second countable group $\Gamma$, a \emph{Lebesgue} $\Gamma$-space is a standard Borel measure space where the $\Gamma$-action preserves only the measure class. Given an equivariant measurable map $p:U \rightarrow V$ between two Lebesgue $\Gamma$-spaces, a \emph{metric along $p$} is a Borel function $d: U \times_p U \rightarrow (0,\infty)$ on the fibered product whose restriction $d_v$ to the fiber $U_v:=p^{-1}(v)$ determines a separable metric space. The $\Gamma$-action is \emph{fiberwise isometric} if any $\gamma \in \Gamma$ acts isometrically on the fibers of $p$, that is $\gamma:U_v \rightarrow U_{\gamma.v}$ is an isometry, and 
$$
d_{\gamma.v}(\gamma.x,\gamma.y)=d_v(x,y) \ ,
$$
for every $\gamma \in \Gamma, v \in V, x,y \in U_v$. A measurable map $q:Y \rightarrow Z$ between Lebesgue $\Gamma$-spaces is \emph{relatively metrically ergodic} if for any fiberwise isometric $\Gamma$-action along a map $p:U \rightarrow V$ and any measurable $\Gamma$-equivariant maps $f:Y \rightarrow U$ and $g:Z \rightarrow V$, there exists a $\Gamma$-equivariant measurable map $\psi:Z \rightarrow U$ such that the following diagram commutes
\begin{equation}\label{eq:lift:diagram}
\xymatrix{
Y \ar[rr]^f \ar[d]^q && U \ar[d]^p \\
Z \ar[rr]^g \ar@{.>}[urr]^\psi && V \ . 
}
\end{equation}

\begin{deft}{\cite[Definition 2.1]{BF14}}\label{def:boundary}
Let $\Gamma$ be a locally compact second countable group. A $\Gamma$-\emph{boundary} is an amenable Lebesgue $\Gamma$-space $B$ such that the projection maps \mbox{$\pi_1:B \times B \rightarrow B$} and $\pi_2:B \times B \rightarrow B$ on the first and the second factor, respectively, are both relatively metrically ergodic. 
\end{deft}

Examples of $\Gamma$-boundaries are Furstenberg-Poisson boundaries for locally compact second countable groups \cite[Theorem 2.7]{BF14}. In particular, a $\Gamma$-boundary for a lattice $\Gamma$ in a semisimple Lie group $G$ can be identified with the quotient space $G/P$, where $P<G$ is a minimal parabolic subgroup. 

Recall that a $\Gamma$-boundary is always a \emph{strong boundary} \cite[Remarks 2.4.(1)]{BF14} in the sense of Burger-Monod \cite{burger2:articolo}.

We are now ready to introduce the notion of Furstenberg map.

\begin{deft}\label{def:furstenberg:map}
Let $\Gamma$ be a countable group and let $(\Omega,\mu)$ be a standard Borel probability $\Gamma$-space. Let $B$ be a $\Gamma$-boundary. Given a virtual dendro-morphism $\sigma:\Gamma \times \Omega \rightarrow \homeo$, we say that a Borel map 
$$
\phi:B \times \Omega \rightarrow \textup{CL}(X) 
$$
is $\sigma$-equivariant if it holds
\begin{equation}\label{eq:borel:equivariance}
\phi(\gamma.b,\gamma.s)=\sigma(\gamma,s)\phi(b,s) \ ,
\end{equation}
for every $\gamma \in \Gamma$ and almost every $b \in B, s \in \Omega$. A \emph{Furstenberg map} for $\sigma$ is a Borel equivariant map $\phi:B \times \Omega \rightarrow X$. 
\end{deft}

Given a $\Gamma$-boundary $B$ and an equivariant Borel map $\phi:B \times \Omega \rightarrow \textup{CL}(X)$, we can define the $s$-\emph{slice} of the map $\phi$ as
$$
\phi_s:B \rightarrow \textup{CL}(X) \ , \ \ \phi_s(b):=\phi(b,s) \ .
$$
By \cite[Chapter VII, Lemma 1.3]{margulis:libro} we have that $\phi_s$ is measurable and by Equation \eqref{eq:borel:equivariance} we obtain that
\begin{equation}\label{eq:equivariance:slice}
\phi_{\gamma.s}(b)=\sigma(\gamma,s)\phi(\gamma^{-1}b) \ ,
\end{equation}
for every $\gamma \in \Gamma$ and almost every $b \in B,s \in \Omega$. 

Before proving the existence of a Furstenberg map for a non-elementary virtual dendro-morphism, we need first to prove the following technical

\begin{lem}\label{lemma:technical}
Let $\Gamma$ be a countable group and let $\Omega$ be an ergodic standard Borel probability space. Let $B$ be a $\Gamma$-boundary. Given a measurable map 
$$
\phi:B \times \Omega \rightarrow \mathcal{P}_2(X) \ ,
$$
we have that
$$
[\phi_s(b)] \cap [\phi_s(b')]=\varnothing \ ,
$$
for almost every $s \in \Omega, b,b' \in B$. 
\end{lem}

\begin{proof}
By the $\Gamma$-ergodicity on the space $B \times B \times \Omega$ \cite[Proposition 2.4]{MonShal0}, we have that either 
$
[\phi_s(b)] \cap [\phi_s(b')]
$
is empty for almost every $s \in \Omega, b,b' \in B$ or it is not trivial. Suppose by contradiction that 
$$
[\phi_s(b)] \cap [\phi_s(b')] \neq \varnothing \ ,
$$
for almost every $s \in \Omega, b,b' \in B$. 

We claim that, for almost every $s \in \Omega$, there exists a full measure subset $A_s \subset B$ such that 
$$
[\phi_s(b)] \cap [\phi_s(b')] \neq \varnothing \ ,
$$
for every $b,b' \in A_s$. By Fubini's Theorem, for almost every $s \in \Omega$, there exists a full measure subset $B_s \subset B$ and, given $b \in B_s$, there exists a full measure subset $B_{b,s} \subset B$ such that
$$
[\phi_s(b)] \cap [\phi_s(b')] \neq \varnothing \ ,
$$
for every $b' \in B_{b,s}$. Now we have two possibilities. If $B_s$ satisfies our claim we are done. If not, there must exist two points $b,c \in B_s$ such that
$$
[\phi_s(b)] \cap [\phi_s(c)] = \varnothing .
$$
We can define $A_s:=B_{b,s} \cap B_{c,s}$. For any $b',c' \in A_s$, by applying \cite[Lemma 2.2]{DM16} to the $4$-tuple of arcs $[\phi_s(b)],[\phi_s(b')],[\phi_s(c)],[\phi_s(c')]$, we get that 
$$
[\phi_s(b')] \cap [\phi_s(c')] \neq \varnothing \ ,
$$
and the claim is proved. 

Thanks to the claim, we can consider the largest set $A_s$ such that 
$$
[\phi_s(b)] \cap [\phi_s(c)] \neq \varnothing \ ,
$$
for every $b,c \in A_s$. Let $\mathscr{A}:=\{A_s\}_{s \in \Omega}$ be the collection of such sets. As a consequence of Equation \eqref{eq:equivariance:slice} and by the maximality assumption, we have that $\mathscr{A}$ is $\Gamma$-invariant, that is $A_{\gamma.s}=\gamma A_s$. For almost every $s \in \Omega$, the intersection 
$$
I_s:=\bigcap_{b' \in A_s} [\phi_s(b')] \neq \varnothing 
$$
is either an arc or a point. The fact that $I_s$ is not empty follows by \cite[Lemma 2.1]{DM16}. The ergodicity of $\Omega$ implies that $I_s$ is a point for almost every $s \in \Omega$ or is an arc for almost every $s \in \Omega$. Being the family $I_s$ equivariant (thanks to the invariance of $\mathscr{A}$ and the equivariance of $\phi$), in both cases we have that $\sigma$ is elementary and this is a contradiction. 
\end{proof}

We are finally ready to prove the existence of a Furstenberg map for a non-elementary cocycle.

\begin{teor}\label{teor:furstenberg:map}
Let $\Gamma$ be a countable group and let $(\Omega,\mu)$ be an ergodic standard Borel probability $\Gamma$-space. Consider a $\Gamma$-boundary $B$. Given a virtual dendro-morphism $\sigma:\Gamma \times \Omega \rightarrow \homeo$, if $\sigma$ is not elementary, then there exists a Furstenberg map
$$
\phi:B \times \Omega \rightarrow X \ .
$$
Additionally, if $\mathscr{M}=\{M_s\}_{s \in \Omega}$ is the equivariant minimal family of subdendrites associated to $\sigma$, then for almost every $s \in \Omega$ the slice $\phi_s$ takes values into the ends of $M_s$, that is 
$$
\phi_s:B \rightarrow \textup{Ends}(M_s) \ .
$$
\end{teor}

\begin{proof}
Since $\sigma$ is not elementary, by Lemma \ref{prop:minimal:family} there exists a minimal family $\mathscr{M}=\{M_s\}_{s \in \Omega}$ for $\sigma$. For almost every $b \in B, s \in \Omega$, define $V_{b,s}:=\textup{Prob}(M_s)$ the space of probabilities on $M_s$. The collection $\mathscr{V}:=\{V_{b,s}\}_{b \in B, s \in \Omega}$ comes with a natural affine action determined by $\sigma$. Since $B$ is amenable, by \cite[Proposition 4.3.4]{zimmer:libro} $B \times \Omega$ is an amenable $\Gamma$-space. Thus there exists an equivariant section to the measurable field $\mathscr{V}$, that is an equivariant map 
$$
\widetilde{\phi}:B \times \Omega \rightarrow \textup{Prob}(X) 
$$
such that $\textup{EssIm}(\widetilde{\phi}_s) \subset \textup{Prob}(M_s)$ for almost every $s \in \Omega$. Here $\textup{EssIm}(\widetilde{\phi}_s)$ denotes the essential image of the slice $\widetilde{\phi_s}$. By composing $\widetilde{\phi}$ with the map given by \cite[Proposition 3.3]{DM16} we obtain a map 
$$
\phi:B \times \Omega \rightarrow \mathcal{P}_{1,2}(X) \ .
$$
By the ergodicity of $B \times \Omega$, we have either $\phi:B \times \Omega \rightarrow X$ or $\phi:B \times \Omega \rightarrow \mathcal{P}_2(X)$. We want to show that the latter case is impossible. 

Suppose by contradiction that we have a map $\phi:B \times \Omega \rightarrow \mathcal{P}_2(X)$. By Lemma \ref{lemma:technical} we know that 
$$
[\phi_s(b)] \cap [\phi_s(b')]=\varnothing \ ,
$$
for almost every $s \in \Omega$ and almost every $b,b' \in B$. We can define the map
$$
f:B \times B \times \Omega \rightarrow \textup{Bund}(X) \ , \ \ f(b,b',s)=(a_s(b),C_s(b,b')) \ ,
$$
where $a_s(b)$ and $C_s(b,b')$ are the unique point and connected component satisfying 
$$
a_s(b) \in [\phi_s(b)] \ , \ \phi_s(b') \subset C_s(b,b') \ , \ C_s(b,b') \cap \phi_s(b)=\varnothing \ .
$$
Notice that we can see $f$ as an element of $\textup{L}^0(\Omega,\textup{Bund}(X))$. Thus, we can apply the relative metric ergodicity of the first projection to get a lift in the following diagram
\begin{equation}\label{diag:lift}
\xymatrix{
B \times B \ar[rr] \ar[d] && \textup{L}^0(\Omega,\textup{Bund}(X))  \ar[d]^{p_\Omega}\\
B \ar[rr] \ar@{.>}[rru] && \textup{L}^0(\Omega,X) \ . 
}
\end{equation}
To apply correctly relative metric ergodicity, we metrize each fiber of the map $p_\Omega$ integrating along $\Omega$ the metrics given by the fiberwise metric along $p:\textup{Bund}(X) \rightarrow X$ (possibly renormalizing them to get bounded metrics, see \cite[Theorem 1]{SS21}). 

The existence of the lift in Diagram \eqref{diag:lift} implies that the component $C_s(b,b')$ does not actually depend on the point $b'$, for almost every $s \in \Omega$. By the way we defined the map $f$, we have that 
$$
\phi_s(b') \subset C_s(b) \ ,
$$
for almost every $b' \in B$. Thus the essential image $E_s$ of $\phi_s$ lies in the closure $\overline{C_s(b)}=C_s(b) \cup a_s(b)$, for almost every $b \in B$. Notice that $\overline{C_s(b)}$ can intersect $\phi_s(b)$ is at most one point. 

Exploiting the equivariance of the map $\phi$ and the fact that $\overline{C_s(b)}$ is closed and connected, the minimality of the family $\{ M_s \}_{s \in \Omega}$ implies that $M_s \subset \overline{C_s(b)}$. Denote by $B'$ the subset of $B$ where $M_s \subset \overline{C_s(b)}$. If we define 
$$
I_s:=\bigcap_{b \in B'} \overline{C_s(b)} \ ,
$$
we have that $M_s \subset I_s$. The latter inclusion leads to a contradiction: in fact $\phi_s(b) \subset M_s$ for almost every $b \in B$, whereas $I_s \cap \phi_s(b)$ is at most one point for almost every $b \in B$. Thus we get that $\phi:B \times \Omega \rightarrow X$ such that almost every $s$-slice takes values in $M_s$, being the latter a minimal family. 

Now we show that each $\phi_s$ has values into $\textup{Ends}(M_s)$. Again, by the ergodicity of the space $B \times B \times \Omega$ we must have that 
$$
\phi_s(b) \neq \phi_s(b') \ ,
$$
for almost every $s \in \Omega$ and almost every $b,b' \in B$ (otherwise $\sigma$ would be elementary). We define the map 
$$
f:B \times B \times \Omega \rightarrow \textup{Bund}(X) \ , \ \ f(b,b',s)=(\phi_s(b),C_s(b,b')) \ ,
$$
where $\phi_s(b') \in C_s(b,b')$. The same argument based on relative metric ergodicity shows that $C_s(b,b')$ does not depend on $b'$ for almost every $s \in \Omega$. Thus $\phi_s(b') \in C_s(b)$ for almost every $b' \in B$ and, by minimality, $M_s \subset \overline{C_s(b)}$. Since $\phi_s(b) \in \textup{Ends}(\overline{C_s(b)})$ and $\phi_s(b) \in M_s$, we must have $\phi_s(b) \in \textup{Ends}(M_s)$. 
\end{proof}

\begin{oss}\label{oss:unicita:essenziale}
Given a minimal family $\mathscr{M}=\{M_s\}_{s \in \Omega}$ for a non-elementary virtual dendro-morphism $\sigma$, by Theorem \ref{teor:furstenberg:map} we know that the $s$-slice of the Furstenberg map takes values in $M_s$. We claim that $\phi$ is essentially unique, that is, given another Furstenberg map $\psi$, we must have
$$
\phi(b,s)=\psi(b,s) \ ,
$$
for almost every $s \in \Omega$ and $b \in B$. Suppose not. By the ergodicity of $B \times \Omega$, we must have
$$
\phi(b,s) \neq \psi(b,s) \ ,
$$
for almost every $s \in \Omega, b \in B$. In this way we obtain an equivariant Borel map 
$$
B \times \Omega \rightarrow \mathcal{P}_2(X) \ , \ \ (b,s) \mapsto \{ \phi_s(b),\psi_s(b) \} \ , 
$$
which would contradict the proof of Theorem \ref{teor:furstenberg:map}. 
\end{oss}

\section{The canonical unitary representation of a virtual dendro-morphism}

In this section we are going to prove our main results. Given a virtual dendro-morphism $\sigma:\Gamma \times \Omega \rightarrow \homeo$, the main step of the proof will be the construction of a canonical unitary $\Gamma$-representation $V$ without invariant vectors such that $\upH^2_b(\Gamma;V)$ contains a non-trivial class. Before giving the details of such construction, we need to recall the cocycle introduced by Monod and Duchesne \cite[Section 9]{DM16}. 

Let $X$ be a dendrite. Given two points $p,q \in X$, we define a Borel function on $\textup{Bund}^2(X)$ as follows
\begin{equation}\label{eq:addition:term}
\alpha(p,q)(a,C,C'):=\begin{cases*}
&1 \text{if $p \in C,q \in C'$ and $C \neq C'$} \ ,\\
-&1 \text{if $p \in C', q \in C$ and $C \neq C'$} \ , \\
&0 \text{otherwise} \ .
\end{cases*}
\end{equation}

Recall that $C,C'$ are two connected components of $X \setminus \{a\}$. As noticed by Duchesne and Monod, the function $\alpha$ is alternating in both $(p,q)$ and $(C,C')$, it remains unchanged by the action of the group $\homeo$ and it is non-zero when the point $a$ lies in the open arc $(p,q)$. Starting from the function $\alpha$, we define an alternating $2$-cocycle taking values in the Borel function on $\textup{Bund}^2(X)$ by setting
$$
\omega_X(p,q,r):=\alpha(p,q)+\alpha(q,r)+\alpha(r,p) \ .
$$
Notice that $\omega_X(p,q,r)$ vanishes at a triple $(a,C,C')$ unless $a$ is the intersection point of the arcs $[p,q],[q,r],[p,r]$. 

By restricting $\omega_X$ to the subbundle $\Lambda(X) \subset \textup{Bund}^2(X)$ defined over the branch points of $X$, Duchesne and Monod \cite[Proposition 9.1]{DM16} proved that 
$$
\omega_X:X^3 \rightarrow \ell^p(\Lambda(X)) 
$$
is a $\homeo$-equivariant alternanting bounded Borel cocycle, where $1 \leq p < \infty$. The space $\ell^p(\Lambda(X))$ is the space of $p$-summable functions on $\Lambda(X)$, which is a separable isometric dual Banach $\homeo$-module. In particular, by \cite[Lemma 3.3.3]{monod:libro} the Borel structures induced by the norm topology, the weak topology and the weak-$^\ast$ topology all coincide. 

Given $1 \leq q < \infty$, we recall the definition of the \emph{Bochner space}
$$
\upL^q(\Omega,\ell^p(\Lambda(X))):=\{ u:\Omega \rightarrow \ell^p(\Lambda(X)) \ | \ \int_\Omega \lVert u(s) \rVert_{\ell^p}^q d\mu(s) < \infty \} . 
$$
In an analogous way, for $q=\infty$ we can set
$$
\upL^\infty(\Omega,\ell^p(\Lambda(X))):=\{ u:\Omega \rightarrow \ell^p(\Lambda(X)) \ | \ \textup{ess sup}_{\Omega} \lVert u(s) \rVert_{\ell^p} < \infty \} \ .
$$ 
For every $1 \leq p < \infty, 1 \leq q \leq \infty$, the Bochner space defined above is a Banach space. Additionally, we have a natural $\Gamma$-action defined by
\begin{equation}\label{eq:action:lp}
((\gamma.u)(s))(a):=u(\gamma^{-1}.s)(\sigma(\gamma^{-1},s)^{-1}(a)) \ ,
\end{equation}
for every $\gamma \in \Gamma, a \in \Lambda(X)$ and almost every $s \in \Omega$. In this way we get a coefficient $\Gamma$-module.

\begin{teor}\label{teor:construction:unitary}
Let $\Gamma$ be a countable group and let $(\Omega,\mu)$ be an ergodic standard Borel probability $\Gamma$-space. Let $\sigma:\Gamma \times \Omega \rightarrow X$ be a non-elementary virtual dendro-morphism. Denote by $V=\upL^\infty(\Omega,\ell^p(\Lambda(X)))$. Then $\textup{H}^2_b(\Gamma;V)$ contains a canonical non-trival element for $1 \leq p  < \infty$. 
\end{teor}

\begin{proof}
Let $B$ a $\Gamma$-boundary. By Theorem \ref{teor:furstenberg:map} there exists a Furstenberg map 
$$
\phi:B \times \Omega \rightarrow X \ , 
$$ 
and by Remark \ref{oss:unicita:essenziale} this map can be chosen canonically. We can now compose $\phi$ with the cocycle $\omega_X$ to get a map
$$
\phi^\ast \omega_X: B^3 \times \Omega \rightarrow \ell^p(\Lambda(X)) \ , 
$$
$$
\phi^\ast \omega_X(b_1,b_2,b_3,s):=\omega_X(\phi_s(b_0),\phi_s(b_1),\phi_s(b_2)) \ ,
$$
that can be viewed as an alternating bounded measurable $\Gamma$-invariant cocycle 
$$
\phi^\ast \omega_X:B^3 \rightarrow \upL^\infty(\Omega,\ell^p(\Lambda(X)))  \ , 
$$
in virtue of the exponential law \cite[Corollary 2.3.3]{monod:libro}. 

By \cite[Theorem 2]{burger2:articolo}, the fact that $B$ is an amenable $\Gamma$-space implies that $\phi^\ast \omega_X$ represents canonically a bounded cohomology class $[\phi^\ast \omega_X]$ in degree two, that is in $\upH^2_b(\Gamma;\upL^\infty(\Omega,\ell^p(\Lambda(X))))$. The alternating property of $\omega_X$ and the ergodicity of $B \times B \times \Omega$ imply that the non-vanishing of $\phi^\ast \omega_X$ guarantees that the class $[\phi^\ast \omega_X]$ is not trivial \cite{burger2:articolo}.

Thus, we are left to show that $\phi^\ast \omega_X$ does not vanish identically. By contradiction, suppose that it is actually trivial. This means that 
$$
\omega_X(\phi_s(b_0),\phi_s(b_1),\phi_s(b_2))=0 \ ,
$$
for almost every $s \in \Omega$ and almost every $b_0,b_1,b_2 \in B$. By Fubini's Theorem we can fix two points $b_0,b_1$ such that 
$$
\omega_X(\phi_s(b_0),\phi_s(b_1),\phi_s(b))=0 \ ,
$$
for almost every $b \in B$. This implies for almost every $s \in \Omega$, the essential image $E_s$ contains at most two points. By the ergodicity of $\Omega$, the cardinality of $E_s$ is essentially constant, so either we have a map $\Omega \rightarrow X$ or a map $\Omega \rightarrow \mathcal{P}_2(X)$. In both cases, the map is $\Gamma$-equivariant by the equivariance of the Furstenberg map. This would imply that $\sigma$ is elementary contradicting the assumptions. 
\end{proof}

We are finally ready to prove 

\begin{repteor}{teor:unitary:rep}
Let $\Gamma$ be a countable group and let $(\Omega,\mu)$ be an ergodic standard Borel probability $\Gamma$-space. Given a non-elementary virtual dendro-morphism $\sigma:\Gamma \times \Omega \rightarrow \homeo$, there exists a canonical unitary $\Gamma$-representation $V$ without invariant vectors and such that $\textup{H}^2_b(\Gamma;V)$ has a non-zero class. 
\end{repteor}

\begin{proof}
Consider the Bochner space $V=\upL^2(\Omega,\ell^2(\Lambda(X)))$. The latter is a Hilbert space with the scalar product defined by
$$
\langle u,v \rangle_V:=\int_\Omega \langle u(s),v(s) \rangle_{\ell^2(\Lambda(X))} d\mu(s) \ .
$$

Since $\Omega$ is a standard Borel space and $\ell^p(\Lambda(X))$ is separable, $V$ is separable as well. The inclusion 
$$
\upL^\infty(\Omega,\ell^2(\Lambda(X))) \rightarrow V \ ,
$$
is injective and adjoint, thus \cite[Corollary 9]{burger2:articolo} implies that the map
$$
\upH^2_b(\Gamma;\upL^\infty(\Omega,\ell^2(\Lambda(X)))) \rightarrow \upH^2_b(\Gamma;V) 
$$ 
is injective. As a consequence the class $[\phi^\ast \omega_X]$ constructed in theorem \ref{teor:construction:unitary} is not trivial when viewed as a class with coefficients in $V$. 

We are left to show that $V$ does not contain any invariant vector. Suppose by contradiction that $u \in V$ is an invariant vector. This means that 
$\gamma.u=u$ for every $\gamma \in \Gamma$. Using Equation \eqref{eq:action:lp} we have that
$$
u(s)(a)=u(\gamma^{-1}.s)(\sigma(\gamma^{-1},s)^{-1}(a)) \ ,
$$
for every $\gamma \in \Gamma, a \in \Lambda(X)$ and almost every $s \in \Omega$. As a consequence we have that the function
$$
s \mapsto \lVert u(s) \rVert_{\ell^\infty} 
$$
is essentially constant, by the ergodicity of $\Omega$. Let $\lambda$ be the essential image. We define
$$
\Lambda_s:=\{ a \in \Lambda(X) \ | \ |(u(s)(a))| =\lambda \} .
$$
Since the function $u(s) \in \ell^2(\Lambda(X))$, the set $\Lambda_s$ is not empty and finite for almost every $s \in \Omega$. 

Let $\pi:\Lambda(X) \rightarrow \textup{Br}(X)$ the bundle projection over the branch points of $X$. We define
$$
L_s:=\pi(\Lambda_s) \ . 
$$
Since the level set $\Lambda_s$ is finite, the same holds for $L_s$. Being finite, we have that $T_s:=[L_s]$ is a tree. Additionally, the $\Gamma$-invariance of $u$ and the $\homeo$-equivariance of the projection $\pi$ imply that 
$$
L_{\gamma.s}=\sigma(\gamma,s)L_s \ ,
$$
and hence 
$$
T_{\gamma.s}=\sigma(\gamma,s)T_s \ .
$$

Since the tree $T_s$ is finite, we can apply Jordan's algorithm \cite{Jor69} to get a unique point or egde: in each step we eliminate the leaves of $T_s$ having a free vertex until we obtain either a unique edge (the center) or a set of egdes joined by a unique point (the center). Thus, applying to each $T_s$ the Jordan's center construction, we get either an equivariant map $\varphi:\Omega \rightarrow X$ or an equivariant map $\varphi:\Omega \rightarrow \mathcal{P}_2(X)$. In this way we get a contradiction to the non-elementarity of $\sigma$. 
\end{proof}

Using the previous theorem we can finally prove our main result. 

\begin{repteor}{teor:elementary}
Let $\mathbf{G}$ be a connected almost $k$-simple algebraic group over a field $\kappa$ with $\textup{rank}_{\kappa}(\mathbf{G}) \geq 2$. Let $\Gamma$ be a lattice the $k$-points of $\mathbf{G}$. Any virtual dendro-morphism of $\Gamma$ over an ergodic standard Borel probability $\Gamma$-space is elementary. 
\end{repteor}

\begin{proof}
Let $\sigma:\Gamma \times \Omega \rightarrow \homeo$ be a virtual dendro-morphism. Suppose by contradiction that $\sigma$ is not elementary. Take the unitary representation $V$ of Theorem \ref{teor:unitary:rep}. Since $V$ has no invariant vectors, by \cite[Theorem 1.4]{MonShal0} we have that
$$
\dim \upH^2_b(\Gamma;V)=0 \ .
$$
The latter is a contradiction to the fact stated in Theorem \ref{teor:unitary:rep} which guarantees the existence of a non-trivial class in degree $2$. This proves the statement and concludes the proof. 
\end{proof}

\bibliographystyle{amsalpha}

\bibliography{biblionote}

\end{document}